\documentclass[psamsfonts,11pt]{amsart}
\usepackage{mathpazo}
\usepackage[top=3cm,bottom=2cm,left=3cm,right=3cm]{geometry}
\usepackage{overpic}
\usepackage{amsmath,amsthm,amssymb,amscd,amsfonts,amsbsy}
\usepackage{color}
\usepackage{hyperref,url}
\usepackage{float}
\usepackage{hyperref}
\usepackage{enumerate}
\usepackage{enumitem}   
\usepackage{mathrsfs}  
\usepackage{mathrsfs}  
\usepackage{amsthm}
\usepackage{amsmath}
\usepackage{amsfonts}
\usepackage{amssymb}
\usepackage{amsthm}
\usepackage{amsmath}
\usepackage{amsfonts}
\usepackage{amssymb}
\usepackage{mathtools}
\usepackage{mathrsfs}
\usepackage{comment}
\usepackage{svg}

\usepackage{amsmath, amssymb, graphics, setspace}
\usepackage[utf8]{inputenc}
\usepackage{listings,xcolor}
\usepackage[english]{babel}

\newcommand{\mathsym}[1]{{}}
\newcommand{\unicode}[1]{{}}

\usepackage[normalem]{ulem}

\newcommand{\R}{\ensuremath{\mathbb{R}}}

\usepackage[toc,page]{appendix}
\usepackage{tikz-cd}

\newtheorem{theorem}{Theorem}[section]
\newtheorem{corollary}[theorem]{Corollary}
\newtheorem{lemma}[theorem]{Lemma}
\newtheorem{proposition}[theorem]{Proposition}

\newtheorem{definition}{Definition}
\newtheorem{remark}{Remark}

\newtheorem{theoremA}{Theorem}

\def\R{\mathbb R}

\newcommand{\vphi}{\varphi}
\newcommand{\cal}[1]{\ensuremath{\mathcal{#1}}}
\newcommand{\set}[1]{\left\{#1\right\}}
\newcommand{\abs}[1]{\left\lvert#1\right\rvert}
\newcommand{\Bigabs}[1]{\Big\lvert#1\Big\rvert}
\newcommand{\biggabs}[1]{\bigg\lvert#1\bigg\rvert}
\newcommand{\norm}[1]{\left\lVert#1\right\rVert}
\newcommand{\hausdorff}[1]{\operatorname{dim_H}\left(#1\right)}

\newcommand{\opname}[1]{\operatorname{#1}}
\renewcommand{\subset}{\subseteq}
\renewcommand{\supset}{\supseteq}
\newcommand{\st}{\text{ such that }}
\newcommand{\ov}[1]{\ensuremath{\overline{#1}}}

\newcommand{\cl}[1]{\ensuremath{\overline{#1}}}

\def\N{\mathbb{N}}
\def\R{\mathbb{R}}

\DeclareMathOperator{\proj}{proj}

\graphicspath{{figures/}}

\title[infinite-piecewise expanding maps]
{Infinite-Piecewise Expanding Maps: \\ Chaos, Ergodicity and Invariant-Set Complexity}
\author{Matheus G. C. Cunha$^1$\and Douglas D. Novaes$^2$ \and Gabriel Ponce$^3$}

\address{$^1$Departamento de Matem\'{a}tica - Instituto de Bioci\^{e}ncias, Letras e Ci\^{e}ncias Exatas (IBILCE) - Universidade
Estadual Paulista (UNESP), \ Rua Crist\'{o}vao Colombo, 2265, Jardim Nazareth, CEP 15054-000, Sao Jos\'{e} do Rio Preto, SP,
Brazil} \email{matheus.gc.cunha@unesp.br}

\address{$^2$Departamento de Matem\'{a}tica - Instituto de Matem\'{a}tica, Estat\'{i}stica e Computa\c{c}\~{a}o Cient\'{i}fica (IMECC) - Universidade
Estadual de Campinas (UNICAMP), \ Rua S\'{e}rgio Buarque de Holanda, 651, Cidade Universit\'{a}ria Zeferino Vaz, CEP 13083-859, Campinas, SP,
Brazil} \email{ddnovaes@unicamp.br}

\address{$^3$Departamento de Matem\'{a}tica - Instituto de Ciências Matemáticas e de Computação (ICMC) - Universidade de São Paulo (USP), \ Avenida Trabalhador São-Carlense, 400, Centro, CEP 13566-590, São Carlos, SP,
Brazil} \email{gaponce@icmc.usp.br}

\usepackage{mathpazo}

\allowdisplaybreaks
\begin{document}

\subjclass[2020]{Primary 37E99, 37A99, 26A18, 28A78, 28A80; Secondary 34A36, 37C29}

\keywords{infinite-piecewise expanding maps, conformal iterated function systems, invariant Cantor sets, Hausdorff dimension, conformal measures, ergodicity, chaotic dynamics}

\maketitle
\begin{abstract} 
In this paper, we study a class of one-dimensional piecewise maps defined by infinitely many smooth expanding branches. This class arises naturally in the context of non-smooth dynamical systems and includes the first-return maps locally defined near sliding Shilnikov connections. By means of the theory of conformal iterated function systems (CIFS), we investigate several dynamical properties of these maps as well as the topological complexity of their invariant sets. In particular, we show that the dynamics restricted to the invariant set is topologically conjugate to the shift on $\mathbb{N}^{\mathbb{N}}$. We also establish the existence of a unique conformal measure that is invariant and ergodic under the map.
 \end{abstract}

\section{Introduction}
\label{section:introduction}

In \cite{PacificoRovellaViana1998InfinitemodalMapsGlobal}, Pacifico et al. investigated a class of one-dimensional infinite-modal maps arising in the study of three-dimensional vector fields exhibiting saddle-focus homoclinic connections. These maps emerge as reduced models of the corresponding first return maps and provide an effective framework for understanding the intricate dynamics generated by such homoclinic phenomena.

A related setting arises in the context of non-smooth dynamical systems. In particular, sliding Shilnikov connections (see Figure \ref{fig:sliding_shilnikov}, left), introduced in \cite{NovaesTeixeira2019ShilnikovProblemFilippov}, are homoclinic connections intrinsic to Filippov systems \cite{Filippov1988DifferentialEquationsDiscontinuous} and can be regarded as the non-smooth counterparts of saddle-focus homoclinic connections. It was shown in \cite{CunhaNovaesPonce2024HausdorffDimensionCantor} that the first-return map associated with a sliding Shilnikov connection is a one-dimensional piecewise map with infinitely many smooth expanding branches, whose domains accumulate at the origin, as illustrated in Figure \ref{fig:sliding_shilnikov}, right.

\begin{figure}[H]
\centering 
\begin{overpic}[width=0.84\linewidth]{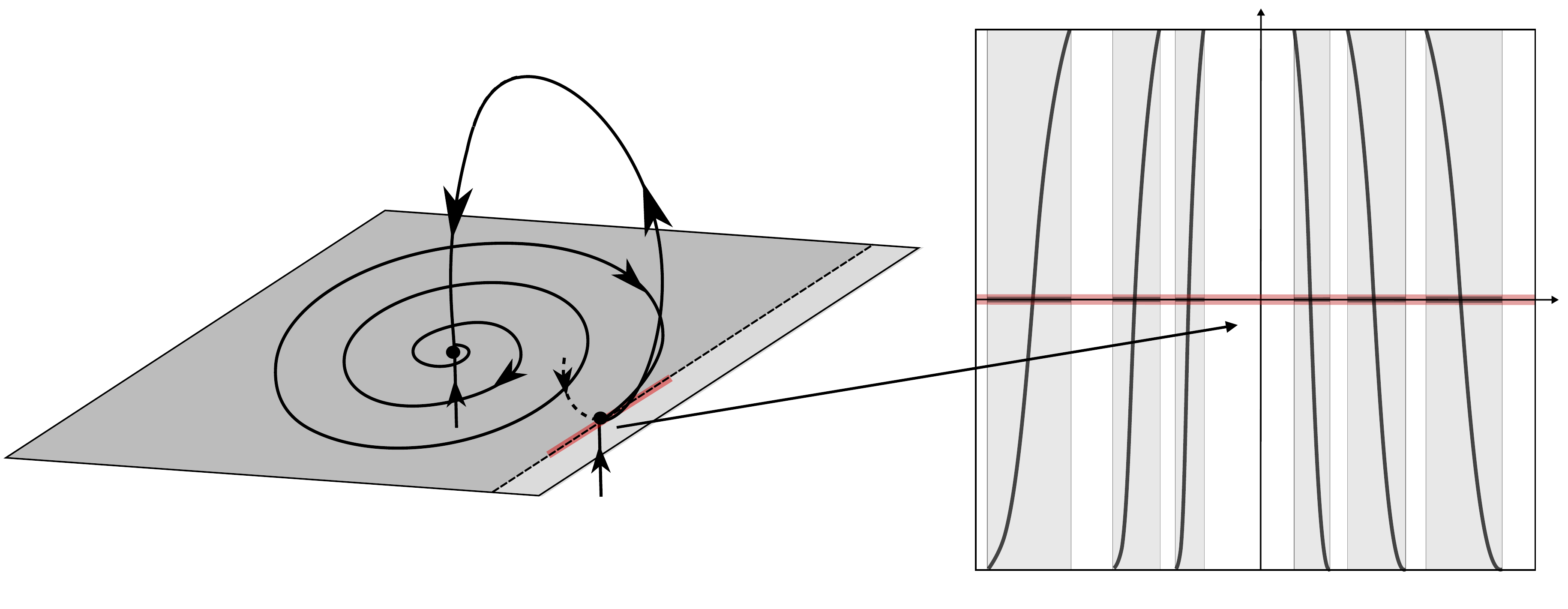}
\put(76,20){\small$\cdots$}
\put(80,20){\small$\cdots$}
\end{overpic}
\vspace{0.2cm}
\caption{Left: Representation of a sliding Shilnikov connection. Right: Representation of the first return map defined around a sliding Shilnikov connection}
\label{fig:sliding_shilnikov}
\end{figure}

This provides the main motivation for the present work. We introduce a general class of piecewise maps that encompasses, as a particular case, the maps arising from sliding Shilnikov connections. The class consists of one-dimensional maps defined by infinitely many smooth expanding branches, which we refer to as {\it infinite-piecewise expanding maps}.

Under suitable assumptions, we investigate several dynamical properties of these maps and the topological complexity of their invariant sets. As we shall see, both are naturally addressed through the theory of conformal iterated function systems.

Our main result, Theorem \ref{theoremA}, not only strengthens the main results of \cite{CunhaNovaesPonce2024HausdorffDimensionCantor} in the context of sliding Shilnikov connections, but also extends them to the broader class of infinite-piecewise expanding maps.

The remainder of this section is devoted to formally introducing the class of infinite-piecewise expanding maps considered in this paper and to stating the main theorem.

\subsection{Definition of the class of infinite-piecewise expanding maps}\label{sec:def}
We consider a map $\vphi: W\subset[-1,1]\to[-1,1]$ satisfying the following conditions:

\begin{enumerate}[label= (H\arabic*), ref = (H\arabic*), start = 0]
\item\label{hypothesis:h0}
 $W = \bigcup_{k\ge1} F_k$, where $\mathcal{J} := \set{F_k}_{k\ge1}$ is a collection of pairwise disjoint  isolated closed intervals, in such way that, for each $J\in\mathcal{J}$, the restriction $\vphi_J := \vphi|_J : J\to[-1,1]$ is a diffeomorphism. By isolated, we mean that there are open intervals $G_k\supset F_k$, $k\geq 1$, with $G_i\cap G_j=\emptyset$ if $i\neq j$.
 \end{enumerate}
 
Taking \ref{hypothesis:h0} into account, for each $J\in\mathcal{J}$, we denote $\psi_J := \vphi_J^{-1} : [-1,1] \to J$ and  $\Psi := \set{\psi_J}_{J\in\mathcal{J}}$. We require some additional conditions for the maps in $\vphi$, which actually will be stated for maps in $\Psi$ for the sake of simplicity.

\begin{enumerate}[label= (H\arabic*), ref = (H\arabic*), start = 1]
\item\label{hypothesis:h1} The family $\Psi$ is uniformly contracting, in the sense that there exists $0<s<1$ such that
\[\abs{\psi_J(x) - \psi_J(y)} \le s\cdot\abs{x-y},\]
for every $J\in \mathcal J$ and all $x,y\in[-1,1]$.

\item\label{hypothesis:h2} There exist constants $L \in \mathbb R$ and $\alpha>0$ such that
\begin{equation*}\label{equation:distortion_hypothesis}
\abs{\psi_J'(x) - \psi_J'(y)} \le L\cdot \inf_{z\in[-1,1]}\abs{\psi_J'(z)}\cdot\abs{x-y}^{\alpha},
\end{equation*}
for every $J\in \mathcal J$ and all $x,y\in[-1,1]$.

\item\label{hypothesis:h3} There exists a parameter $q \in (1, +\infty]$ and a constant $C \ge 1$ such that, for every finite exponent $p \le q$, there exist a constant $C_p \ge 1$ and an integer $N_p>0$  satisfying
    \begin{equation*}
        \frac{1}{C n^q} \le \sup_{z \in [-1,1]} \abs{\psi_{F_n}'(z)} \le \frac{C_p}{n^p},\,\, \text{for all}\,\, n \ge N_p.
    \end{equation*}
    Here, we are adopting the conventions $1/q = 0$ and $1/(C n^q) = 0$ when $q = +\infty$.
    
\end{enumerate}

Conditions \ref{hypothesis:h1} and \ref{hypothesis:h2} are standard. They impose a necessary uniformity on the contraction rates of the inverse branches and require their derivatives to be uniformly H\"older continuous. If the system consisted of only a finite number of branches, these conditions would naturally reduce to the requirement that each restriction to a branch is a strict contraction with a H\"older continuous derivative.

Condition \ref{hypothesis:h3} is a little less intuitive than the others and essentially states that the derivatives $\psi'_{F_n}(z)$ have polynomial (or super-polynomial) decay as $n\to +\infty$, where the parameter $q$ captures the supremum of valid polynomial decay rates. In terms of the original map $\vphi$, this is equivalent to saying that if $\vphi^{-1}(z) = \set{z_1, z_2, z_3, \ldots}\subset [-1,1]$, then the sequence of derivatives $\set{\abs{\vphi'(z_n)}}_{n\ge1}$ has polynomial (or super-polynomial) growth. 

\begin{remark}\label{remark}
Proposition \ref{prop:series_bounds} in Appendix \ref{section:appendixA} shows that \ref{hypothesis:h3} implies convergence of the series 
\begin{equation*}
    \sum_{J\in\cal{J}}\left(\sup_{z\in[-1,1]}\abs{\psi_{J}'(z)}\right)^t
\end{equation*}
for all $t>1/q$.
In particular, since $q > 1$, by taking $t = 1 > 1/q$ in Proposition \ref{prop:series_bounds}, we get
\begin{equation}\label{equation:convergence_hypothesis}
    \sum_{J\in\cal{J}}\left(\sup_{z\in[-1,1]}\abs{\psi_{J}'(z)}\right) <+\infty.
\end{equation}
For several of our geometric bounds, we will only require the convergence condition in \eqref{equation:convergence_hypothesis}, rather than the stronger assumption in \ref{hypothesis:h3}. However, \ref{hypothesis:h3} is technically more accessible, allows for explicit quantitative bounds on the Hausdorff dimension in our main result, and has a clearer dynamical interpretation, which is why we retain it among the main hypotheses.

Corollary \ref{corollary:exponential_bounds} in Appendix \ref{section:appendixA}  also implies that there exists an integer $N \ge 1$ such that
\[
    \sum_{\eta\in\cal{J}^N}\left(\sup_{z\in[-1,1]}\abs{\psi_{\eta}'(z)}\right) < 1,
\]
where, for $\eta = (\eta_1, \eta_2, \ldots, \eta_N)\in\cal{J}^N$, we define $\psi_{\eta} := \psi_{\eta_1}\circ\cdots\circ \psi_{\eta_N}$. The proof of this statement is given in Lemma \ref{lemma:exponential_decay_measure} of the Appendix \ref{section:appendixA} .

\end{remark}

\subsection{Statement of the main result}
Consider the invariant set of $\varphi$ given by
\begin{equation}
\label{equation:invariant_set}
\Lambda := \set{w\in W: \vphi^k(w)\in  [-1,1], \text{ for all } k\ge0} = \bigcap_{k\ge0}\vphi^{-k}[-1,1].
\end{equation}
Other sets that will be of interest are the {\it singular set}
\begin{equation}\label{equation:Z_1}
Z_1 := \ov{W}\setminus W	
\end{equation}
and the {\it spread singular set}
\begin{equation}\label{equation:Z}
Z := \bigcup_{k\ge0}\vphi^{-k}(Z_1).
\end{equation}

Our main result provides a topological understanding of the invariant set and its closure, as well as insights into their complexity in terms of Hausdorff dimension  $\hausdorff{\cdot}$  and one-dimensional Lebesgue measure as $\mu(\cdot)$. It also characterizes the dynamics of $\vphi|_{\Lambda}$ via a topological conjugacy with the shift map on $\mathbb{N}^{\mathbb{N}}$. Moreover, we establish the existence of a probability measure in $\Lambda$, invariant and ergodic regarding $\vphi$, that distinguishes the images of $\psi_{J_i}, \psi_{J_j}\in\Psi$ (for $i\neq j$) and possesses a quantitative relationship with the Hausdorff dimension of the invariant set.

\begin{theoremA}\label{theoremA}
Consider the infinite-piecewise expanding map
\[
\vphi: W\subset[-1,1]\to[-1,1]
\]
satisfying conditions \ref{hypothesis:h0}--\ref{hypothesis:h3}.  Let $\Lambda$ be its invariant set as defined in \eqref{equation:invariant_set} and $Z_1$ and $Z$ as defined in \eqref{equation:Z_1} and \eqref{equation:Z}, respectively. Then, the following statements hold.
\begin{enumerate}[label= {(\alph*)}, ref = A(\alph*)]
	
	\item\label{theorem:geometry_of_invariant_set}
	Let $q \in (1, +\infty]$ be the parameter as defined in \ref{hypothesis:h3}. Then,
	\[
    \frac{1}{q} < \dim_H(\Lambda) < 1.
	\]
	Furthermore, $\mu(\Lambda) = 0$.
	
	\item\label{theorem:structure_of_closure_of_lambda}
	The closure of the invariant set is given by the disjoint union $\ov{\Lambda}=\Lambda\,\dot{\cup}\,Z$.
	
	\item\label{theorem:geometry_of_closure_of_lambda}
	$\hausdorff{\ov{\Lambda}} = \max\set{\hausdorff{\Lambda}, \hausdorff{Z_1}}$ and $\mu(\ov{\Lambda}) = \mu(Z)$.
	
	\item\label{theorem:cantor_set}
	$\ov{\Lambda}$ is a Cantor set.
 
 	\item\label{theorem:topological_conjugation}
	The map $\vphi|_\Lambda: \Lambda\to\Lambda$ is topologically conjugate to the shift on $\N^\N$.
	
	\item\label{theorem:conformal_measure}
	There exists a probability measure $m$ on $\Lambda$ that is $\hat{t}$-conformal, for $\hat{t} := \hausdorff{\Lambda}$. Moreover, there exists a unique measure $m^* \sim m$ which is invariant and ergodic under $\vphi$. \end{enumerate}
\end{theoremA}

The methods developed in this work are sufficiently general to admit a wide range of applications. For example, in \cite{CunhaNovaesPonce2024HausdorffDimensionCantor}, the local first-return map associated with a sliding Shilnikov connection in a Filippov system is shown to satisfy the hypotheses of Theorem \ref{theoremA}. Consequently, it provides conclusions concerning geometric, topological, and dynamical properties of the corresponding local invariant set. In fact, items (a)–(d) of Theorem \ref{theoremA} were already established in \cite{CunhaNovaesPonce2024HausdorffDimensionCantor} for the sliding Shilnikov setting. The present framework, however, allows these results to be strengthened by additionally guaranteeing properties (e) and (f) of Theorem \ref{theoremA}.

The first-return map associated with a sliding Shilnikov connection provides perhaps the simplest example of an infinite-piecewise expanding map in the class considered here, with $Z_1$ consisting of a single point. Nevertheless, the set $Z_1$ can exhibit a much richer structure. While it is straightforward to construct examples in which $Z_1$ is finite or countably discrete, Appendix \ref{section:appendixB} presents a first instance where $Z_1$ is the triadic Cantor set. Moreover, it is shown there that the same construction can be adapted to generate examples in which $Z_1$ has positive Hausdorff dimension and even positive Lebesgue measure.

\subsection{Structure of the paper} In Section \ref{section:cifs}, we introduce the aspects of the theory of iterated function systems that are relevant to our analysis. In Section \ref{section:dynamics_via_cifs}, we establish the connection between the infinite-piecewise expanding maps introduced earlier and the theory of conformal iterated function systems. In Section \ref{subsection:closure_of_lambda} we look into the topological structure of the closure of the invariant set $\Lambda$. Finally, the proofs of the statements of Theorem \ref{theoremA} are presented in Section \ref{section:proofs}. 

\section{Conformal iterated function systems \emph{(CIFS)}}
\label{section:cifs}

The setting of our main tool for our studies consists of an at most countable collection of contractions on a compact set $K = \ov{K^{\circ}} \subset \R^d$. This set of functions $\mathcal{F} := \set{f_i : K\to K}_{i\in I}$, with the condition that, for all $x, y \in K$ and for each $i \in I$, there exists a constant $s_i < 1$ such that
\[
\abs{f_i(x) - f_i(y)}\le s_i\cdot\abs{x-y},
\]
is called an \emph{iterated function system (IFS)}.

For $k \ge 1$, we can consider (finite) sequences of the form $\eta = (\eta_1,\ldots,\eta_k) \in \bigcup_{j\ge1} I^j$ and associate the function $f_{\eta}$ by
\[
f_{\eta} := f_{\eta_1}\circ\cdots\circ f_{\eta_k}.
\]
Also, for $\omega \in \cal{J}^k$ and $\eta \in \cal{J}^l$, we define $\omega\eta\in\cal{J}^{k+l}$ as the juxtaposition and it induces the composition $\psi_\omega \circ\psi_\eta$.

Assuming an uniform contraction rate $s < 1$, we can define a projection map $\proj : I^\N \to K$ for every (infinite) sequence $\omega = (\omega_1, \omega_2, \omega_3, \ldots) \in I^\N$ by
\begin{equation*}
\set{\proj(\omega)} = \bigcap_{k\ge1}f_{\omega|_k}(K),
\end{equation*}
where $\omega|_k := (\omega_1, \ldots, \omega_k)$ for each $k \ge1$. Since each $f_i$ is a contraction with a uniform bound $s < 1$, the resulting sequence of sets forms a nested collection of non-empty compact subsets with diameters approaching zero, so  the intersection is a single point --- therefore $\proj$ a well-defined map.

Using these notions, we say that the \emph{attractor set} of an IFS is:
\begin{equation}\label{equation:atractor_set}
\Delta := \proj\left(I^{\N}\right) = \bigcup_{\omega\in I^{\N}}\set{\proj(\omega)} = \bigcup_{\omega\in I^{\N}} \bigcap_{k\ge1}f_{\omega|_k}(K).
\end{equation}
This set has the invariance property $\Delta = \bigcup_{i\in I}f_i(\Delta)$, and it is the largest set satisfying it.

The theory of infinite IFS has many differences from its finite counterpart. This motivates the introduction of the theory of \emph{conformal iterated function systems (CIFS)}, a generalization of the theory that deals with infinite IFS having well-behaved properties and geometry. To study this, from now on, let $\cal{F} := \set{f_i}_{i\in I}$ be an IFS satisfying the following conditions:
\begin{enumerate}[label= (C\arabic*), ref = (C\arabic*)]
    \item\label{condition:injection}
    The functions $f_i:K\to K$ are injective, for all $i\in I$, ;
    
    \item\label{condition:unifcontract}
    There exists $s<1$ such that
    \[
    \abs{f_i(x) - f_i(y)} \le s\cdot\abs{x - y}, \quad \text{for any } x,y\in K \text{ and for all }i\in I;
    \]
    in other words, they are uniformly contractive on $K$;
    
    \item\label{condition:osc}
    (Open Set Condition) The set $K$ is connected, and satisfies both $f_i(K^{\circ})\subset K^{\circ}$ and $f_i(K^{\circ})\cap f_j(K^{\circ})=\emptyset$, for all $i,j\in I, i\neq j$;
    
    \item\label{condition:conformal}(Conformal Property)
    There exists $V=V^{\circ}\subset\R^d$ containing $K$, such that each one of the maps $f_i$ extends to a function $f_i^V$, a $\cal{C}^{1,\varepsilon}$ diffeomorphism on $V$. Besides, these maps are conformal, meaning they satisfy $Df_i^V(x) = \rho_{x,i}\cdot\opname{Isom}_{x,i}$, where $\rho_{x,i}\in\R$ and $\opname{Isom}_{x,i}:\R^d\to\R^d$ is an isometry, for any $x\in V$ and for all $i\in I$;

	\item\label{condition:cone_alt}
	The inequality
	\[
	\inf_{x\in\partial K}\inf_{0<r<1}\dfrac{m_d(B_r^{\circ}(x)\cap K^{\circ})}{m_d(B_r^{\circ}(x))} >0,
	\]
	holds, with $m_d$ being the $d$-dimensional Lebesgue measure;

    \item\label{condition:bdp_alt}
    There exist constants $L \ge 1$ and $\alpha > 0$ such that, for every $i\in I$,
    \[
    \biggabs{\norm{Df_i^V(x)} - \norm{Df_i^V(y)}} \le L \cdot \norm{(Df_i^V)^{-1}}_{\infty}^{-1} \cdot \abs{x-y}^{\alpha},
    \]
    where $x,y\in V$ (the open set from Condition \ref{condition:conformal}). Here, $\norm{\cdot}$ denotes the operator norm, and we also have the uniform norm $\norm{(Df_i^V)^{-1}}_{\infty} := \sup_{z \in V} \norm{(Df_i^V(z))^{-1}}$.
\end{enumerate}

It can be proven that Condition \ref{condition:bdp_alt} implies the following:
\begin{enumerate}[label=(BDP), ref=(BDP),  leftmargin=\widthof{(BDP)}]
    \item\label{condition:bdp}(Bounded Distortion Property)
    There exists $M>1$ such that
    \begin{equation*}
    \dfrac{1}{M} \le \dfrac{\norm{Df_{\eta}(x)}}{\norm{Df_{\eta}(y
    )}} \le M,	
    \end{equation*}
    for every $\eta\in \bigcup_{j\ge1} I^j$ and every $x,y\in V$ (again the same set from Condition \ref{condition:conformal});
\end{enumerate}
the proof can be found in \cite[Lemma 2.2]{MauldinUrbanski1996DimensionsMeasuresInfinite}.

These are known collectively as the \textit{conformal conditions}. As a reference to them, we refer to \cite{MauldinUrbanski1996DimensionsMeasuresInfinite, Mauldin1995InfiniteIteratedFunction}.

\subsection{Pressure of a CIFS}

Regarding conformal systems, the so-called \emph{pressure} of the system dictates many geometric properties of the attractor set (see, for instance, \cite[Chapter 3]{MauldinUrbanski1996DimensionsMeasuresInfinite} and \cite[Chapter 6]{Mauldin1995InfiniteIteratedFunction}). Here is the definition:

\begin{definition}[Pressure of a CIFS]
\label{definition:pressure}
For each $k\ge1$, consider the auxiliary functions $P_k:[0,+\infty)\to\R\cup\set{+\infty}$ given by
\[
P_k(t) := \sum_{\eta\in I^k}\norm{Df_{\eta}}_{\infty}^t.
\]

The \emph{pressure of the system} is the function $P:[0,+\infty)\to\R\cup\set{+\infty}$ defined as
\[
P(t) := \lim_{k\to+\infty}\dfrac{1}{k}\log P_k(t)= \lim_{k\to+\infty}\dfrac{1}{k}\log\left(\sum_{\eta\in I^k}\norm{Df_{\eta}}_{\infty}^t\right).
\]
\end{definition}

Now we discuss and list some properties of the auxiliary functions $P_k$ and the pressure function $P$ that will be used.

 Firstly, we remark that $P_1(t)$ is a non-increasing function, implying that the set
 \[
 \cal{P} := \set{t \ge0 \st P_1(t) < +\infty}
 \]
 is an interval of the form $\cal{P} = [\zeta,+\infty)$ or $\cal{P} = (\zeta,+\infty)$. From now on, we will always denote $\zeta := \inf \cal{P}$; also, we call $\cal{P}$ the \emph{domain of finiteness}.

\begin{proposition}[{\cite[Theorem 6.1 and Theorem 6.2]{Mauldin1995InfiniteIteratedFunction}}]
\label{proposition:pressure_function}
The auxiliary functions $P_k$ and the pressure function $P$ satisfy the following properties:
\begin{enumerate}[label = \ref{proposition:pressure_function}(\alph*), ref = \ref{proposition:pressure_function}(\alph*), 
]
    \item\label{proposition:p_k}
    For all $k\ge1$, the functions $P_k(t)$ are non-increasing in the interval $[0,+\infty)$, strictly decreasing in $[\zeta,+\infty)$, and continuous in $\cal{P}$.
    \item\label{proposition:p}
    Similarly, the pressure function $P(t)$ is non-increasing in $[0,+\infty)$, strictly decreasing in $[\zeta,+\infty)$, and continuous in $\cal{P}$.
    \item\label{proposition:pressure_bounds}
    For $M>1$ being the constant from Condition \ref{condition:bdp}, the following bounds are satisfied:
    \[
    -t\log M + \log P_1(t) \le P(t) \le \log P_1(t).
    \] 
\end{enumerate}
\end{proposition}

The pressure function also satisfies an upper bound related to the auxiliary functions $P_k$.

\begin{lemma}\label{lemma:subadditivity}
    For any $t \ge 0$, the sequence $k \mapsto \log P_k(t)$ is sub-additive. Consequently, the pressure function satisfies
    \[
        P(t) \le \frac{1}{k} \log P_k(t), \quad \text{for all } k \ge 1.
    \]
\end{lemma}

\begin{proof}
    Let $\omega \in \cal{J}^k$ and $\eta \in \cal{J}^l$ and, by the chain rule, we have $\norm{\psi_{\omega\eta}'}_{\infty} \le \norm{\psi_{\omega}'}_{\infty} \norm{\psi_{\eta}'}_{\infty}$. This then implies
    \[
        P_{k+l}(t) = \sum_{\omega\in\cal{J}^k, \eta\in\cal{J}^l} \norm{\psi_{\omega\eta}'}_{\infty}^t \le \left(\sum_{\omega\in\cal{J}^k}\norm{\psi_{\omega}'}_{\infty}^t\right)\left(\sum_{\eta\in\cal{J}^l}\norm{\psi_{\eta}'}_{\infty}^t\right) = P_k(t)P_l(t).
    \]
    So $\log P_{k+m}(t) \le \log P_k(t) + \log P_m(t)$, which is to say that $\log P_k(t)$ is sub-additive. By applying Fekete's Sub-additive Lemma we get
    \[
        P(t) := \lim_{n\to\infty} \frac{1}{n}\log P_n(t) = \inf_{n\ge 1} \frac{1}{n}\log P_n(t).
    \]
    Therefore, $P(t) \le \frac{1}{k} \log P_k(t)$ for all integers $k \ge 1$.
\end{proof}

If the pressure function has a zero, additional properties may be deduced about the CIFS, motivating the next definition:

\begin{definition}
A CIFS is said to be \emph{regular} if its pressure function has a unique zero. For the unique $\hat{t}\ge1$ with $P(\hat{t})=0$, we call the CIFS $\hat{t}$-regular.
\end{definition}

The following result characterizes the Hausdorff dimension of the invariant set of a $\hat{t}$-regular CIFS.

\begin{proposition}[{\cite[Theorem 7.4]{Mauldin1995InfiniteIteratedFunction}}]
\label{proposition:pressure_hausdorff_dimension}
Let $\Delta$ be the invariant set of a  $\hat{t}$-regular CIFS. Then, $\hausdorff{\Delta} = \hat{t}$.
\end{proposition}

\section{The dynamics of infinite-piecewise expanding maps via CIFS}
\label{section:dynamics_via_cifs}

As discussed in Section \ref{section:introduction}, each one of the functions $\psi_J$ is a contraction, with a uniform contraction constant $s < 1$. This qualifies $\Psi := \set{\psi_J : [-1,1]\to J}_{J\in\mathcal{J}}$ as an IFS. It is, nonetheless, more than that:

\begin{theorem}
The set $\Psi$ is a CIFS.
\end{theorem}

\begin{proof}
Conditions \ref{condition:injection}, \ref{condition:unifcontract}, and \ref{condition:osc} follow directly from the defining properties of the class of infinite-piecewise expanding maps introduced in Section \ref{sec:def}. Moreover, a straightforward computation shows that Condition \ref{condition:cone_alt} is also satisfied on the set $[-1,1]$.

For Condition \ref{condition:conformal}, it is enough to take $V := (-1-\delta, 1+\delta)$, for some $\delta>0$, and extend each function $\psi_J$ to $\psi_J^V:V\to V$ by simply extending the functions on each side of the interval in a affine way with slopes equal to the lateral derivatives of $\psi_J$ at $-1$ and $1$. This is well-defined, since $\psi_J[-1,1]\subset[-1,1]$ and the extension is still a contraction, therefore $\psi_J^V(V)\subset V$. Besides, as $0 < \abs{(\psi_J^V)'(x)} < 1$, for all $x\in[-1,1]$, each function is Lipschitz continuous and, consequently, $\psi_J^V$ is also $\cal{C}^{1,\varepsilon}$. Finally, in the one-dimensional context, every diffeomorphism is trivially conformal, since the derivative at each point can be seen as just a number multiplying the identity map.

At last, for Condition \ref{condition:bdp_alt}, note that we have
\begin{gather*}
\norm{{\left(\psi'_{J}\right)}^{-1}}_{\infty}^{-1} = \left(\sup_{x\in[-1,1]}\abs{(\psi'_J(x))^{-1}}\right)^{-1}  = \left(\dfrac{1}{\inf_{x\in[-1,1]}\abs{\psi_J'(x)}}\right)^{-1} = \inf_{x\in[-1,1]}\abs{\psi_J'(x)},
\end{gather*}
and applying the condition \ref{hypothesis:h2}, we have
\begin{gather*}
\Bigabs{\abs{\psi'_{J }(x)} - \abs{\psi'_{J }(y)}} \le \abs{\psi'_{J }(x) - \psi'_{J }(y)} \le L\cdot \inf_{z\in[-1,1]}\abs{\psi'_{J }(z)}\cdot \abs{x-y}^{\alpha} \\ 
= L \cdot \norm{(\psi'_{J })^{-1}}_{\infty}^{-1} \cdot \abs{x-y}^{\alpha},
\end{gather*}
for some $L>1, \alpha>0$ and for all $x,y\in[-1,1]$. Then, extending the inequality for $\psi_J^V$ and $x,y\in V$, since by construction the derivatives $(\psi_J^V)'$ in $V$ have values already attained by $\psi_J'$ in $[-1,1]$, the IFS satisfies \ref{condition:bdp_alt}.

Therefore, $\Psi$ is an IFS satisfying all the conditions to be conformal.
\end{proof}

The behavior of the IFS is, naturally, linked to the dynamics of the map $\vphi$. In fact, the invariant set $\Lambda$ and the attractor set $\Delta$ are the same. To prove this, we need the following lemmas.

\begin{lemma}[{\cite[Theorem 3.1]{Mauldin1995InfiniteIteratedFunction}}]
\label{lemma:switch_union_intersection}
For a CIFS, the following relationship holds:
\[
\Delta = \bigcup_{\omega\in I^{\N}} \bigcap_{k\ge1}f_{\omega|_k}(K) = \bigcap_{k\ge1}\bigcup_{\eta\in I^k}f_{\eta}(K).
\]
\end{lemma}

From what follows, it will help to define the auxiliary sets
\begin{equation}\label{equation:lambda_k_and_delta_k}
\Lambda_k := \bigcap_{1\le i\le k}\vphi^{-i}[-1,1],\text{ for }k\ge1,\,\, \text{ and }\,\, \Delta_k := \bigcup_{\eta\in \cal{J}^k}\psi_{\eta}[-1,1],\text{ for } k\ge 1.
\end{equation}

Note that
\begin{equation}\label{equation:lambda_k_is_nested}
\Lambda_{k+1} \subset \Lambda_k \quad\text{ and }\quad \Delta_{k+1} \subset \Delta_k.
\end{equation}

Also, we have the recursive property
\begin{equation}\label{equation:recursive_lambda_k}
\Delta_{k+1} = \bigcup_{\eta\in\cal{J}^{k+1}}\psi_\eta[-1,1] = \bigcup_{J\in\cal{J}}\psi_J(\Delta_k), \quad \text{for any $k\ge1$}.
\end{equation}

\begin{lemma}\label{lemma:lambda_k_equals_delta_k}
We have that $\Lambda_k = \Delta_k$, for all $k\ge1$.
\end{lemma}
\begin{proof}
Indeed
\begin{equation*}
\Lambda_1 = \vphi^{-1}[-1,1] = \cup_{J\in\cal{J}}\psi_J[-1,1] = \Delta_1,
\end{equation*}
(which also means $\Lambda_1 = \Delta_1 = W$), and also
\begin{align*}
&\phantom{\iff}\ \ x\in \Lambda_k \\
&\iff \vphi_{J_k} \circ \vphi_{J_{k-1}} \circ \ldots \circ \vphi_{J_1}(x) = y, \text{ for some }y\in [-1,1]   \text{ and } J_i\in\cal{J}\text{ for }1\le i\le k\\
&\iff \psi_{J _k}^{-1} \circ \psi_{J_{k-1}}^{-1}\circ \ldots \circ \psi_{J_1}^{-1}(x) = y, \text{ for some }y\in[-1,1]  \text{ and } J_i\in\cal{J}\text{ for }1\le i\le k \\
&\iff x = \psi_{J_1} \circ \ldots \circ \psi_{J _{k-1}} \circ \psi_{J _k}(y), \text{ for some }y\in[-1,1]  \text{ and } J_i\in\cal{J}\text{ for }1\le i\le k \\
&\iff x \in \psi_{\eta}[-1,1], \text{ for some }\eta\in \cal{J}^k  \\
&\iff x\in \Delta_k.
\end{align*}

\end{proof}

Now we can prove the following proposition.
\begin{proposition}\label{proposition:lambda_equals_delta}
The invariant set $\Lambda$ of $\vphi$ and the attractor set $\Delta$ of the CIFS $\Psi$ are the same.
\end{proposition}
\begin{proof}
Note that the invariant set $\Lambda$ of $\vphi$, as expressed in \eqref{equation:invariant_set}, can be written as a nested intersection of the sets given in \eqref{equation:lambda_k_and_delta_k}:
\begin{equation}\label{equation:intersection_lambda_k}
\Lambda =\bigcap_{k\ge1}\vphi^{-k}[-1,1]=\bigcap_{k\ge1}\Lambda_k.
\end{equation}
Furthermore, by Lemma \ref{lemma:switch_union_intersection}, the attractor set of  the CIFS $\Psi$ can by viewed as
\begin{equation}\label{equation:intersection_delta_k}
\Delta = \bigcap_{k\ge1}\bigcup_{\eta\in \cal{J}^k}\psi_{\eta}[-1,1] = \bigcap_{k\ge1}\Delta_k.
\end{equation}

Thus, by applying Lemma \ref{lemma:lambda_k_equals_delta_k} in \eqref{equation:intersection_lambda_k} and \eqref{equation:intersection_delta_k}, we may finally conclude that $\Lambda = \Delta$.
\end{proof}

\section{The Invariant Set $\Lambda$}\label{subsection:closure_of_lambda}

An important object of our results is the closure of the invariant set, so now we turn our attention to $\ov{\Lambda}$, using the characterization $\Lambda=\Delta=\bigcap_{k\ge 1} \Delta_{k}$ given by Proposition \ref{proposition:lambda_equals_delta} to analyze its structure and properties. 

Since we aim to understand the closure of the invariant set, a first step is to study the closure of $W = \Delta_1$ and the effects of $\Psi$ on it. For that, we define
\begin{equation}\label{equation:Z_k}
Z_k := \left(\bigcup_{j=1}^{k-1}\bigcup_{\eta\in\cal{J}^j}\psi_{\eta}(Z_1)\right)\cup Z_1 = \left(\bigcup_{j=1}^{k-1}\vphi^{-j}(Z_1)\right) \cup Z_1,\text{ for } k\ge1,
\end{equation}
where $Z_1$ is defined as in \eqref{equation:Z_1}. Note that
\begin{equation}\label{equation:Z_k_nested}
Z_{k+1} \supset Z_k, \text{ for all } k\ge 1
\end{equation}
Also, they satisfy the recursive property
\begin{equation}\label{equation:recursive_Z_k}
Z_{k+1} = \left(\bigcup_{J\in\cal{J}}\psi_J(Z_k)\right)\cup Z_1, \quad \text{for any $k\ge1$}.
\end{equation}

We can also characterize the set $Z$ in \eqref{equation:Z} as
\begin{equation}\label{equation:Z_alt}
Z := \bigcup_{k\ge0}\vphi^{-k}(Z_1) = \bigcup_{k\ge 1} Z_k.
\end{equation}

We finally define
\[
\Delta_Z := \bigcap_{k\ge1}\left(\Delta_k \cup Z_k\right).
\]

Clearly, since $\Delta_k \subset \Delta_k \cup Z_k$, we must have $\Delta \subset \Delta_Z$. Now we are going to prove that $\Delta_Z$ is compact and $\ov{\Delta}=\Delta_Z$, and for that, we will need a property of $Z_1$ that arises from the structure of the sequence $\set{F_n}_{n\ge1}$, so we state the following lemma.

\begin{lemma}
\label{lemma:equiv_isolated_and_cofinite}
The following statements are equivalent:
\begin{enumerate}
	\item Each $F_i$ is isolated from the others;
	\item $Z_1$ is a non-empty compact set and, for any open set $V_{Z_1}$ containing $Z_1$, there are only finitely many intervals $J\in\cal{J}$ such that $J\not\subset V_{Z_1}$.
\end{enumerate}
\end{lemma}

\begin{proof}
For the first implication, we recall that since each $F_i$ is isolated, there are open intervals $G_i\supset F_i$, with $G_i\cap G_j=\emptyset$ if $i\neq j$.
	
	First, taking a sequence $x_i\in F_i$, it must have a convergent subsequence, so let us write $x_{i_k} \to x$. Clearly, $x\in\ov{W}$, and since the intervals are isolated, $x\notin W$, otherwise there would be an accumulation of points near the interval where $x$ lies, contradicting the fact that they are isolated. So $Z_1\neq\emptyset$.
	
	We also note that $Z_1 := \ov{W}\setminus W = \ov{W}\setminus (\cup_i G_i)$. Indeed, if $x\in \overline W$ and $x\in G_i$, then $x$ cannot belong to $Z_1$, since if $x\in F_i$ then $x\in W$, and if $x\in G_i\setminus F_i$ there exists a neighborhood of $x$ disjoint from $W$, contradicting $x\in\overline W$. Also, $\ov{W}\setminus (\cup_i G_i)$ is the difference between a closed set and an open set, therefore it is closed and compact.
	
	Finally, take an open set $\cal{V}_{Z_1}$ of $Z_1$, and let us consider the index set $\cal{J}\setminus\cal{J}_{Z_1}$, where $\cal{J}_{Z_1} := \set{J\in\cal{J} : J\subset V_{Z_1},\text{for some } V_{Z_1}\in\cal{V}_{Z_1}}$. If there exists infinite intervals $J_{n_k}\in \cal{J}\setminus\cal{J}_{Z_1}$, then there must exists a convergent sequence $y_{n_k}\to y$, with $y_{n_k} \in F_{n_k}$. Again, the isolation of the intervals implies $y\notin W$, hence $y\in\ov{W}$ and $y\notin W$, or equivalently $y\in Z_1$. But that is an absurd since it must lie inside the open set $V_{Z_1}$, so it cannot be a limit point of intervals that are not contained in $V_{Z_1}$.

For the reciprocal statement the proof is by the contrapositive, so let us suppose that there is an accumulation interval, that is an interval for which any open set, there are infinitely many other intervals converging to it, say $\tilde{F}$; besides, let us assume that $Z_1$ is a non-empty compact set, and we will show that it must be the case that there exists an open set $V_{Z_1}$ such that the number of intervals $J\in\cal{J}$ not contained in $V_{Z_1}$ is infinite.
	
	Since $Z_1$ and $\tilde{F}$ are non-empty compact sets, we can cover $Z_1$ by an open set $V_{Z_1}$ with positive distance from $\tilde{F}$.
	
	Therefore, there must be an infinite number of intervals converging to $\tilde{F}$ (since it is an accumulation interval) that are not contained in the open set $V_{Z_1}$ (since it has positive distance from $\tilde{F}$).
\end{proof}

Now we are going to prove two propositions that taken together imply that each level set $Z_{k}$ is exactly the points outside $\Delta_k$ needed to ``close'' the level set $\Delta_k$.

\begin{proposition}\label{proposition:compacity}

The sets $\Delta_k \cup Z_k$, for $k\ge 1$, and $\Delta_Z$ are compact.
\end{proposition}
\begin{proof}
The proof is by induction on $k$.

For $k=1$, we have that $\Delta_1\cup Z_1 = W\cup\left(\ov{W}\setminus W\right) = \cl{W}$, which is compact.

Now, for the inductive hypothesis, we assume that $\Delta_k \cup Z_k$ is compact, for some  $k\ge1$, to show that $\Delta_{k+1} \cup Z_{k+1}$ is compact too.

First, applying the properties in \eqref{equation:recursive_lambda_k} and \eqref{equation:recursive_Z_k}, we get that
\begin{align*}
\Delta_{k+1} \cup Z_{k+1} &= \left(\bigcup_{J\in\cal{J}}\psi_J(\Delta_k)\right) \cup \left(\bigcup_{J\in\cal{J}}\psi_J(Z_k)\right)\cup Z_1 \\
\\
&= \left(\bigcup_{J\in\cal{J}}\psi_J (\Delta_k\cup Z_k)\right) \cup Z_1.
\end{align*}

So let $\cal{V}$ be an open cover of $\Delta_{k+1} \cup Z_{k+1}$ with an open subcover $\cal{V}_{Z_1} \subset \cal{V}$ covering $Z_1$. Note that $\cal{V}_{Z_1}$ may be taken finite, since $Z_1$ is compact.

By Lemma \ref{lemma:equiv_isolated_and_cofinite}, the index set $\cal{J}\setminus\cal{J}_{Z_1}$ is finite, and the inductive hypothesis gives us that each one of $\psi_J(\Delta_{k} \cup Z_{k})$ is compact, therefore we must have that
\[
\bigcup_{J\in\cal{J}\setminus\cal{J}_{Z_1}}\psi_J(\Delta_{k} \cup Z_{k})
\]
is compact, being a finite union of compact sets and so having a finite subcover $\tilde{\cal{V}}\subset\cal{V}$.

Therefore  
$\tilde{\cal{V}}\cup\cal{V}_{Z_1}\subset\cal{V}$ is a finite subcover of $\Delta_{k+1} \cup Z_{k+1}$, so it is compact.

In addition, since it is an intersection of compact sets, $\Delta_Z$ is compact.
\end{proof}

\begin{proposition}\label{proposition:delta_k_disjoint_Z_k}
We have $\Delta_k\cap Z_k = \emptyset$, for all $k\ge1$.
\end{proposition}
\begin{proof}
The proof is by induction on $k$.

For the base case $k=1$, we simply observe that $\Delta_1\cap Z_1 = \Lambda_1\cap (\cl{W}\setminus W) = W\cap (\cl{W}\setminus W) = \emptyset$.

Now assume that for some $k\ge1$ we have $\Delta_k\cap Z_k=\emptyset$, and let us consider $\Delta_{k+1}\cap Z_{k+1}$. Then
\begin{align*}
\Delta_{k+1}\cap Z_{k+1} &= \left[\bigcup_{J\in\cal{J}}\psi_J(\Delta_k)\right]\cap\left[\left(\bigcup_{J\in\cal{J}}\psi_J(Z_k)\right)\cup Z_1\right] \\
&= \bigcup_{J\in\cal{J}}\Bigg[\Bigg(\psi_J(\Delta_k)\cap\psi_J(Z_k)\Bigg)\cup \Bigg(\psi_J(\Delta_k)\cap Z_1\Bigg)\Bigg]
\end{align*}

The fact that each $\psi_J$ is injective and the induction hypothesis give us that $\psi_J(\Delta_k)\cap\psi_J(Z_k) = \psi_J(\Delta_k\cap Z_k) = \psi_J(\emptyset) = \emptyset$, and since $\psi_J(\Delta_k)\subset J\subset W$, we must have $\psi_J(\Delta_k)\cap Z_1 =\emptyset$.
\end{proof}

As said before, Propositions \ref{proposition:compacity} and \ref{proposition:delta_k_disjoint_Z_k} mean that each $Z_{k}$ is exactly the collection of points needed to close $\Delta_k$, allowing one more characterization
\[
Z_k = \cl{\Lambda_k}\setminus\Lambda_k.
\]

\begin{proposition}
We have $\ov{\Delta} = \Delta_Z$.
\end{proposition}
\begin{proof}
First, given that $\Delta\subset\Delta_Z$ and $\Delta_Z$ is compact by Proposition \ref{proposition:compacity}, we must have $\ov{\Delta}\subset\Delta_Z$.

Thus it remains to prove that $\Delta_Z\subset\ov{\Delta}$. For that, let us consider $x\in\Delta_Z = \bigcap_{k\ge1}\left(\Delta_k \cup Z_{k} \right)$, meaning that $x\in \Delta_k \cup Z_{k}$  for all $k\ge1$.

If $x\in \Delta_k$, for all $k\ge1$, then we actually have that $x\in\Delta\subset\ov{\Delta}$ and there is nothing to prove.

Otherwise, let us assume $x\in Z_{\kappa}$, for some minimal $1\le \kappa\le k-1$.

For the case $\kappa=1$, we have $x\in Z_1 = \ov{W}\setminus W$. This implies that there are a sequence of indexes $i_n$ with $x = \lim_{n\to +\infty}x_{i_n}$, for any $x_{i_n}\in F_{i_n}$. Then by taking a sequence $\set{y_{i_n} : y_i\in\Delta\cap F_{i_n}}$ we deduce that $x\in\ov{\Delta}$.

Now let us assume that $\kappa\ge 2$, which is to say that there exists $\ov\eta = (J_1,\ldots,J_{\ell})\in\cal{J}^{\ell}$, for some $2\le\ell\le \kappa$, such that $x \in \psi_{\ov\eta}(\tilde{x}) = \psi_{(J_1,\ldots,J_{\ell})}(\tilde{x})$, with $\tilde{x}\in Z_1$. Take $\set{x_i}_{i\ge1}\subset\Delta$ a sequence converging to the point $\tilde{x}$. By the continuity of $\psi_{\ov\eta}$, we deduce that
\[
x = \psi_{\ov\eta}(\tilde{x}) = \psi_{\ov\eta}\left(\lim_{i\to +\infty}x_i\right) = \lim_{i\to +\infty}\psi_{\ov\eta}(x_i),
\]
and given that $\Delta$ is $\vphi$-invariant, we conclude that $x$ the limit of a sequence of elements of $\Delta$, therefore $x\in\ov{\Delta}$.
\end{proof}

\section{Proofs of the main result}
\label{section:proofs}

Since many of the results concern the properties of CIFS, Hausdorff dimension and Lebesgue measure, we present some of their properties that we are going to use.

\begin{proposition}[{\cite[Theorem 2(2)]{Schleicher2007HausdorffDimensionIts}}]\label{proposition:stability}
If $\set{S_k}_{k\ge1}$ is a countable collection of sets, then we have $\hausdorff{\bigcup_{k\ge1}S_k} = \sup_{k\ge1}\set{\hausdorff{S_k}}$.
\end{proposition}

\begin{proposition}[{\cite[Theorem 2(5)]{Schleicher2007HausdorffDimensionIts}}]
\label{proposition:hausdorff_invariance}
If $f$ is a bi-Lipschitz map, then $\hausdorff{S} = \hausdorff{f(S)}$, which implies that diffeomorphic compact sets have the same Hausdorff dimension.
\end{proposition}

\begin{proposition}[{\cite[Theorem 2(8)]{Schleicher2007HausdorffDimensionIts}}]
\label{proposition:lebesgue_zero}
If $\hausdorff{S} < 1$, then the one-dimensional Lebesgue measure of $S$ is $0$.
\end{proposition}

We also prove one more proposition:

\begin{proposition}\label{prop:regularity}
    The CIFS $\Psi$ is regular; specifically, there exists a unique $1/q < \hat{t} < 1$ such that $P(\hat{t}) = 0$.
\end{proposition}

\begin{proof}
    By Proposition~\ref{prop:series_bounds}, we have
    \[
    \lim_{t \to 1/q^+} P_1(t) = +\infty,
    \]
    and using the bound 
    \[
    -t\log M + \log P_1(t) \le P(t)
    \]
    in Proposition~\ref{proposition:pressure_bounds}, we get
    \[
    \lim_{t \to 1/q^+} P(t) = +\infty.
    \]
    
    On the other hand, there exists $N\ge1$ with
    \[
    P(1) \le \dfrac{1}{N}\log P_N(1) < 0	
    \]
    (see Lemma~\ref{lemma:subadditivity} and Corollary~\ref{corollary:exponential_bounds}).
    
    Since $P(t)$ is continuous and strictly decreasing on $(1/q, +\infty)$, we have a unique root $1/q < \hat{t} < 1$ satisfying $P(\hat{t}) = 0$.
\end{proof}

\subsection{Proof of Theorem \ref{theorem:geometry_of_invariant_set}}

We are now in position to prove Theorem \ref{theorem:geometry_of_invariant_set}.

This is just an application of Propositions~\ref{proposition:pressure_hausdorff_dimension} and Proposition~\ref{prop:regularity}  to  conclude the Hausdorff dimension bounds
    \[
    1/q < \dim_H(\Lambda) < 1.
    \]
    
   It immediately follows from Proposition~\ref{proposition:lebesgue_zero} that $\mu(\Lambda) = 0$.

\subsection{Proof of Theorem \ref{theorem:structure_of_closure_of_lambda}}

Using the characterization of the invariant set $\Lambda$ as the attractor set $\Delta$ enabled by Proposition \ref{proposition:lambda_equals_delta}, we can use the CIFS context, where we may have better tools and notations. With this in mind, we are now going to prove Theorem \ref{theorem:structure_of_closure_of_lambda}:

Now we are going to prove that $\ov{\Delta} = \Delta\ \cup\ Z$.

First, we must show $\ov{\Delta}\subset  \Delta\cup Z$, and for that we take $x\in\ov{\Delta} = \bigcap_{k\ge1}\left(\Delta_k\cup Z_{k}\right)$, that is, $x\in\Delta_k\cup Z_{k}$, for all $k\ge1$. If $x\in\Delta_k$, for all $k\ge1$, by \eqref{equation:intersection_delta_k}, we deduce that $x\in\Delta\subset \Delta\cup Z$. Otherwise, for some $\ell\ge 1$ and by  \eqref{equation:Z_alt}, we must have $x\in Z_{\ell}\subset Z\subset \Delta\cup Z$. Therefore, $\ov{\Delta} \subset \Delta \cup Z$.

Now, the opposite inclusion, $ \Delta\cup   Z\subset \ov{\Delta}$. Let $x\in \Delta \cup Z$, and given that $\Delta\subset\ov{\Delta}$, we may consider only the case $x\in Z$. Taking \eqref{equation:Z_alt} into account, let $\kappa\ge1$ 
be the first non-negative integer with  $x\in Z_\kappa$.

If $\kappa=1$, then by applying \eqref{equation:Z_k_nested} we get $x = z\in Z_1 \subset Z_k$, for all $k\ge1$, so $x \in\bigcap_{k\ge1}\left(\Delta_k\cup Z_{k}\right) =\ov{\Delta}$.

Then, let us assume $\kappa\ge2$ and, again by applying \eqref{equation:Z_k_nested}, we deduce it must be the case that
\begin{equation}\label{equation:x_in_Z_kappa}
x\in Z_\kappa\subset Z_k, \text{ for every } k\ge\kappa \implies x\in\cup_{k\ge\kappa}Z_k
\end{equation}
so we just need to prove that $x\in \Delta_k$ for every $1\le k\le\kappa-1$.

This comes from applying \eqref{equation:Z_k} to get that $x\in Z_\kappa \implies x\in\bigcup_{\eta\in\cal{J}^\ell}\psi_{\eta}(Z_1)$, for some $1\le \ell\le\kappa-1$; on the other hand, if $\ell<\kappa$, then $x\in Z_{\ell}$ contradicts the minimality of $\kappa\ge2$, therefore
\begin{equation*}
x\in\bigcup_{\eta\in\cal{J}^\kappa}\psi_{\eta}(Z_1) \iff x\in \psi_{\ov{\eta}}(Z_1), \text{ for some } \ov{\eta} = (J_1, \ldots, J_\kappa)\in\cal{J}^\kappa.
\end{equation*}
This is the same to say that $x\in \psi_{J_1}\circ\cdots\circ\psi_{J_\kappa}(Z_1)\subset \psi_{J_1}\circ\cdots\circ\psi_{J_\kappa}[-1,1]\subset\Delta_\kappa$. But by the nested property in \eqref{equation:lambda_k_is_nested} this means that
\begin{equation}\label{equation:x_in_delta_k}
x\in \Delta_k, \text{ for } 1\le k\le \kappa \implies x\in\cup_{1\le k\le\kappa}\Delta_k;
\end{equation}
clearly \eqref{equation:x_in_Z_kappa} and \eqref{equation:x_in_delta_k} imply $x\in \Delta_k\cup Z_{k}$ for every $k\ge 1$, that is, $x\in \bigcap_{k\ge1}\left(\Delta_k\cup Z_{k}\right)=\ov{\Delta}$.

Furthermore, the union must be disjoint, since $\Delta$ is $\vphi$-invariant, $\Delta\subset\Delta_1$ and $\Delta_1\cap Z_1=\emptyset$, so we must have $\Delta\cap Z=\emptyset$.

\subsection{Proof of Theorem \ref{theorem:geometry_of_closure_of_lambda}}
\label{subsection:proof_Ac}

We now may quickly prove Theorem \ref{theorem:geometry_of_closure_of_lambda}:

Firstly, since $\ov{\Lambda} = \Lambda\, \cup\, Z$, we deduce $\hausdorff{\ov{\Lambda}} = \max\set{\hausdorff{\Lambda}, \hausdorff{Z}}$, by Proposition \ref{proposition:stability}. In fact, by applying the bi-Lipschitz property of all maps in $\Psi$ and Proposition \ref{proposition:hausdorff_invariance} many times, we can even conclude that $\hausdorff{Z} = \hausdorff{Z_1}$, thus $\hausdorff{\ov{\Lambda}} = \max\set{\hausdorff{\Lambda}, \hausdorff{Z_1}}$.

By Theorem \ref{theorem:geometry_of_invariant_set}, we have $\hausdorff{\Lambda} < 1$.
Then, by applying Proposition \ref{proposition:lebesgue_zero}, we deduce that $\mu(\Lambda) = 0$, which together with Theorem \ref{theorem:structure_of_closure_of_lambda} implies $\mu(\ov{\Lambda}) = \mu(Z)$.

\subsection{Proof of Theorem \ref{theorem:cantor_set}}

A \emph{Cantor set} is a non-empty metric space $F$  which is compact, totally disconnected, and perfect. Topologically speaking, the Cantor set is unique, since any two of them have a homeomorphism to each other (see, for instance, \cite[Chapter 2, Theorem 67 and Theorem 73]{Pugh2015RealMathematicalAnalysis}).

Now, we are ready to prove Theorem \ref{theorem:cantor_set}:

We will show that $\ov{\Lambda}$ has each one of the properties that defines a Cantor set.

Clearly $\ov{\Lambda}$ is a closed subset of the compact space $[-1,1]$, so it is compact.

Next, we show that $\ov{\Lambda}$ is totally disconnected. For that, we recall that a subset of $\R$ is totally disconnected if and only if it contains no non-empty intervals. Let us define the open set
\begin{equation}\label{eq:G}
G := (-1,1) \setminus \ov{W},
\end{equation}
which is non-empty because each $F_k$ is isolated by an open set $G_k$. Note that $\Lambda\subset W \subset\ov{W}\implies G\cap\Lambda=\emptyset$, and also since $\Lambda$ is invariant, $\psi_{\eta}(V)\cap\Lambda =\emptyset$ for any $\eta\in\cal{J}^k$, with $k\ge1$, and for any subset $V \subset G$. Finally, let $x \in \ov{\Lambda}$ and let $U$ be an arbitrary open neighborhood of $x$. We consider two cases based on the decomposition $\ov{\Lambda} = \Lambda \,\dot{\cup}\, Z$ given by Theorem \ref{theorem:structure_of_closure_of_lambda}:
\begin{itemize}
	\item If $x \in \Lambda$, there exists an infinite sequence $\omega \in \cal{J}^\N$ with $x = \proj(\omega)$. Because the CIFS $\Psi$ is uniformly contractive by \ref{hypothesis:h1}, the diameter of the interval $\psi_{\omega|_k}[-1,1]$ converges to $0$ as $k \to \to +\infty$. Thus, for a sufficiently large $k_U$, this interval is contained in $U$. Now we consider a non-empty open interval $V\subset G$. The diffeomorphic image $\psi_{\omega|_{k_U}}(V)$ is a non-empty open interval inside $U$, disjoint from $\Lambda$ and consequently from $\ov{\Lambda}$.

	\item If $x \in Z$, by Theorem \ref{theorem:structure_of_closure_of_lambda}, there exists an integer $k_x \ge 0$ such that $x \in \vphi^{-k_x}(Z_1)$, which is to say that $x = \psi_{\eta}(z)$ for some word $\eta \in \cal{J}^{k_x}$ and some $z \in Z_1$. Given that $x\in U$, then the diffeomorphic image $\psi_{\eta}^{-1}(U)$ contains an open interval $I$ with $z\in I$. Since $z \in Z_1 = \ov{W} \setminus W$, we have that $I$ must intersect some interval $F_j\subset W$ with $z\notin F_j$, and since $I$ is connected, we have $I\cap\partial F_j\neq\emptyset$. This, together with the fact that $F_j\subset G_j\subset G$, implies that $\tilde{G}_j := G_j \setminus F_j \subset G$ is an open set satisfying $I\cap \tilde{G}_j\neq\emptyset$, so there is a non-empty open interval $V\subset I\cap\tilde{G}_j\subset G$. Again, by taking the diffeomorphic image $\psi_{\eta}(V)$, we obtain a non-empty open interval inside $U$, disjoint from $\Lambda$ and consequently from $\ov{\Lambda}$.
\end{itemize}
In either case, every open neighborhood of every point in $\ov{\Lambda}$ contains a non-empty open interval completely disjoint from $\ov{\Lambda}$. This implies that $\ov{\Lambda}$ cannot contain any non-empty interval, so it is totally disconnected.

To show $\ov{\Lambda}$ is perfect, we choose an arbitrary $x\in\Lambda$, and from \eqref{equation:atractor_set}, we take its representation by the projection map $\ov{\eta}=(J_1, J_2, J_3, \ldots)\in\cal{J}^{\N}$. In other words, we write $x=\proj(\ov{\eta})$ (recall that this is well-defined due to the uniformly contractive property \ref{hypothesis:h1} satisfied by the CIFS $\Psi$). Choosing a sequence of elements $\set{x_k}_{k\ge1}\subset\Lambda$ satisfying $x_k\in \psi_{(J_1,\ldots,J_k)}[-1,1]$ and $x_k\not\in\psi_{(J_1,\ldots,J_k,J_{k+1})}[-1,1]$, for $k\ge1$, we assure that the sequence satisfies both $\set{x_k}_{k\ge 1}\subset \Lambda\setminus\set{x}$ and $x_k\to x$, so $x$ is an accumulation point of $\Lambda$. Since $x$ was arbitrary and the closure of a set with no isolated points cannot contain isolated points, we deduce that $\ov{\Lambda}$ is a perfect set.

\subsection{Proof of Theorem \ref{theorem:topological_conjugation}}

Let $I$ be a countable index set, finite or infinite. We now look at the symbolic dynamical system $I^\N$ with the so-called \textit{shift map} $\sigma:I^\N\to I^\N$ defined by $\sigma(i_1, i_2, i_3, \ldots) := (i_2, i_3, i_4, \ldots)$.

This set can be endowed with the metric $\opname{dist}:I^\N\times I^\N \to [0,+\infty)$ defined as
\[
\opname{dist}(a,b) = \opname{dist}((a_n)_{n\ge1}, (b_n)_{n\ge1}) := \begin{cases}
	0, &\quad\text{if } a=b; \\
	\quad\\
	\dfrac{1}{\kappa}, &\quad\text{if } a\neq b,
	\end{cases}
\]
where $\kappa:=\min\set{j\in\N \st a(j)\neq b(j)}$, and in this space $\sigma$ is a continuous function. This system is an example of a \emph{(topological) Bernoulli system}.

If we make a topological conjugation of a system to a Bernoulli system, we gain a lot of information. Thus our aim in the following proposition is to prove that the dynamical system $(\Lambda, \vphi)$  and a Bernoulli system are topologically conjugated to each other.

Consider the dynamical system $(\cal{J}^\N, \sigma)$. In what follow, we are going to show that the following diagram commutes:
\[
\begin{tikzcd}[column sep=8em, row sep=6em ,every label/.append
style={font=\normalsize}]
\cal{J}^{\N} \arrow{r}{\sigma} \arrow{d}[swap]{\proj} & \cal{J}^{\N} \arrow{d}{\proj} \\
\Lambda \arrow{r}[swap]{\vphi|_\Lambda} & \Lambda
\end{tikzcd}
\]
Furthermore, the function $\proj:\cal{J}^{\N}\to\Lambda$ is an homeomorphism.

For a given $\omega = (J_1, J_2, J_3, \ldots)\in\cal{J}^{\N}$, we have that
\begin{gather*}
\proj \circ\ \sigma(\omega) = \proj \circ\ \sigma(J_1, J_2, J_3, \ldots) \\
= \proj(J_2, J_3, J_4, \ldots) = \psi_{(J_2, J_3, J_4, \ldots)}[-1,1].
\end{gather*}

On the other hand, we also have
\begin{gather*}
\vphi\circ\proj(\omega) = \vphi\circ\proj(J_1, J_2, J_3, \ldots) = \vphi\circ\psi_{(J_1, J_2, J_3, \ldots)}[-1,1] \\
= \vphi \circ \underbrace{\psi_{J_1}}_{\vphi_{J_1}^{-1}}\circ\ \psi_{(J_2, J_3, J_4, \ldots)}[-1,1]
= \psi_{(J_2, J_3, J_4, \ldots)}[-1,1].
\end{gather*}

So the diagram commutes.

Now, to prove that the projection map $\proj$ is an homeomorphism:
 
\begin{itemize}
	\item $\proj$ is continuous:
	
	Take arbitrary $a, b\in\cal{J}^{\N}$, with $a\neq b$ and $\opname{dist}(a,b)<\epsilon$.
	Then there exists $\kappa=\kappa(\epsilon)\ge1$, such that $a(i) =  b(i)$, for all $1\le i\le\kappa$.
	
We have
\[
\proj(a) = \psi_a[-1,1] = \psi_{a(1)}\circ\ldots\circ\psi_{a(\kappa)}\left(\psi_{\sigma^{\kappa+1}(a)}[-1,1]\right)
\]
and
\begin{align*}
\proj(b) = \psi_ b[-1,1] &= \psi_{ b(1)}\circ\ldots\circ\psi_{ b(\kappa)}\left(\psi_{\sigma^{\kappa+1}(b)}[-1,1]\right) \\
&= \psi_{a(1)}\circ\ldots\circ\psi_{a(\kappa)}\left(\psi_{\sigma^{\kappa+1}(b)}[-1,1]\right),
\end{align*}
and so
\begin{gather*}
	\Bigabs{\proj(a) - \proj(b)} \le \sup_{x,y\in[-1,1]}\Bigabs{\psi_{(a(1),\ldots,a(\kappa))}(x) - \psi_{(a(1),\ldots,a(\kappa))}(y)}\\
	\le s^{\kappa}\cdot\sup_{x,y\in[-1,1]}\abs{x-y}\le \opname{diam}[-1,1]\cdot s^{\kappa},
\end{gather*}
which guarantees the continuity of the projection map.
	
	\item $\proj^{-1}$ is continuous: 
	
	Choose arbitrary $x,y\in\Lambda$, with $x\neq y$ and $\abs{a-b}<\epsilon$.  So there exists $\kappa = \kappa(\epsilon)\ge1$ such that $x,y\in\psi_{J_1}\circ\ldots\circ\psi_{J_\kappa}[-1,1]$.
	
	Therefore, we have \[
	\proj^{-1}(x) = (J_1,\ldots,J_{\kappa}, J_{\kappa+1}^x, J_{\kappa+2}^x,\ldots)\in\cal{J}^{\N}
	\]
	and 
	\[
	\proj^{-1}(x) = (J_1,\ldots,J_{\kappa}, J_{\kappa+1}^y, J_{\kappa+2}^y,\ldots)\in\cal{J}^{\N},\]
	for some $\set{J_i}_{1\le i\le\kappa}\subset\cal{J}$ and $\set{J_i^x,J_i^y}_{i\ge\kappa+1}\subset\cal{J}$.
	
	From this, we deduce that $\opname{dist}\left(\proj^{-1}(x), \proj^{-1}(y)\right) \le \dfrac{1}{\kappa}$, proving that $\proj^{-1}$ is also continuous.
	\end{itemize}

Since trivially $\left(\mathcal{J}^{\N}, \sigma\right) \cong \left(\N^\N, \sigma\right)$, we have the desired result.

Therefore, topologically the dynamical system $(\Lambda, \vphi|_\Lambda)$ has the same behavior of the symbolic dynamical system $(\cal{J}^\N, \sigma)$, from which we may deduce topological properties. As an example, we may find that the topological entropy of $(\Lambda, \vphi|_\Lambda)$ is infinite, and its set of periodic orbits is dense.

\subsection{Proof of Theorem \ref{theorem:conformal_measure}}

In Subsection \ref{subsection:proof_Ac}, it is showed that the one-dimensional Lebesgue measure of the invariant set $\Lambda$ is zero. However, other measures may, in principle, convey information about the dynamics of the map $\vphi$. So now we are going to se some results of the theory of regular CIFS, in order to apply them to the local invariant set and the first return map.

Now, we are going to explore some definitions and results to properly apply some geometric results of the theory of regular CIFS to our context:  

\begin{definition}[$t$-Conformal Measures]
\label{definition:t_conformal_measures}

Consider a CIFS $\cal{F}$ and $t\ge1$ a positive real number. A measure $m_t$ is called a \emph{$t$-conformal measure (for the CIFS $\cal{F}$)} if
\begin{enumerate}
    \item $m_t(\Delta) = 1$;
    \item $m_t(f_i(B)) = \int_{B}\norm{Df_i}^tdm_t$, for all $i\in I$ and for all borelian sets $B$;
    \item $m_t(f_i(K)\cap f_j(K)) = 0$, for all $i\in I$, $i\neq j$.
\end{enumerate}
\end{definition}

The next proposition shows us how the regularity of a CIFS and the existence of a $t$-conformal measure for the CIFS are closely related:

\begin{proposition}[{\cite[Theorem 7.2]{Mauldin1995InfiniteIteratedFunction}}]
\label{proposition:existence_conformal_measure}
A CIFS $\cal{F}$ is $\hat{t}$-regular if and only if there exists a $\hat{t}$-conformal measure $m_{\hat{t}}$ for $\cal{F}$.
\end{proposition}

For the following result, we need to introduce some definitions and notations: for each $\eta = (\eta_1\ldots,\eta_k)\in \bigcup_{j} I^j$, we refer to $\abs{\eta} := k$ as being the \emph{length of $\eta$}. We also define the set $[\eta] := \set{\omega\in I^\N \st \omega|_{\abs{\eta}} = \eta}$, which is called the \emph{cylinder generated by $\eta$}.

We then have the following: 
\begin{proposition}[{\cite[Lemma 3.6 and Theorem 3.8]{MauldinUrbanski1996DimensionsMeasuresInfinite}}]\label{proposition:measure_on_index_set}
Let $\cal{F}$ be a $\hat{t}$-regular CIFS, and consider the $\hat{t}$-conformal measure $m_{\hat{t}}$ given by Proposition \ref{proposition:existence_conformal_measure}.

Then there exists a unique Borel probability measure $\nu$ on $I^\N$ with
\[
\nu([\eta]) = \int_K\norm{Df_{\eta}}^{\hat{t}}dm_{\hat{t}},
\]
for all $\eta\in\cup_{j\ge 1}I^j$.

Besides, there is a unique measure $\nu^*$ on $I^\N$, with the property of being invariant under the shift $\sigma$ and being equivalent to $\nu$.

Moreover, $\dfrac{1}{M^{\hat{t}}} \le \dfrac{d\nu^*}{d\nu} \le M^{\hat{t}}$, with $M$ being the constant from Condition \ref{condition:bdp}.
\end{proposition}

When each point of the attractor set $\Delta$ is given by a unique $\omega\in I^\N$ by the means of the projection map, the measure $\nu^*$ on the coding space naturally associates to a measure on the attractor set itself. This association happens to interact with the dynamics induced by the CIFS. Now, consider $T:\Delta\to\Delta$ defined by $T(\proj(\omega)) := \proj(\sigma(\omega))$.

\begin{proposition}[{\cite[Theorem 8.2]{Mauldin1995InfiniteIteratedFunction}}]
\label{proposition:shift_invariant}
Let $\cal{F}$ be a $\hat{t}$-regular CIFS. Then $m_{\hat{t}} = \nu\circ\proj^{-1}$, and the measure $m_{\hat{t}}^* := \nu^*\circ \proj^{-1}$ on $\Delta$ is the unique invariant and ergodic measure with respect to $T$ and satisfying $m_{\hat{t}}^* \sim m_{\hat{t}}$ on $\Delta$.
\end{proposition}

A detailed construction of all of these measures can be found in  \cite[Section 3]{MauldinUrbanski1996DimensionsMeasuresInfinite} and \cite[Section 8]{Mauldin1995InfiniteIteratedFunction}.

In our context this result provides the following proof of Theorem \ref{theorem:conformal_measure}:

\begin{proof}[Proof of Theorem \ref{theorem:conformal_measure}]
By applying Proposition \ref{proposition:existence_conformal_measure}, we prove the existence of a $\hat{t}$-conformal measure $m$, for some $\hat{t}\in(0,1)$.

Finally, by Proposition \ref{proposition:measure_on_index_set} and Proposition \ref{proposition:shift_invariant} we deduce the existence of a unique measure $m^*$, invariant and ergodic relative to $\vphi$, satisfying $m^* \sim m$.
\end{proof}

\appendix

\section{Implication of Condition \ref{hypothesis:h3}} \label{section:appendixA}

The following proposition establishes the relationship between \ref{hypothesis:h3} and the convergence of $P_1(t)$ (see Definition \ref{definition:pressure}).

\begin{proposition}\label{prop:series_bounds}
    Assume that the family $\Psi = \set{\psi_{F_n}}_{n \in \N}$ satisfies \ref{hypothesis:h3} with parameter $q \in (1, +\infty]$. Then the series
    \begin{equation}\label{equation:convergence_of_series}
        P_1(t) := \sum_{n=1}^\infty \left(\sup_{z \in [-1,1]} \abs{\psi_{F_n}'(z)}\right)^t
    \end{equation}
    converges for any $t > 1/q$ and diverges for any $0 \le t \le 1/q$. Furthermore, $\lim_{t \to 1/q^+} P_1(t) = +\infty$.
\end{proposition}

\begin{proof}
    We denote $a_n := \sup_{z \in [-1,1]} \abs{\psi_{F_n}'(z)}$, and analyze separately each case.
    
    \medskip\noindent
    \emph{Convergence for $t > 1/q$:}
    
    Because $t > 1/q$, we can always choose a finite exponent $p \le q$ such that $pt > 1$ (if $q < +\infty$ we choose $p = q$, otherwise if $q = +\infty$ we choose any finite $p > 1/t$).
    
    By \ref{hypothesis:h3}, there exist $C_p \ge 1$ and $N_p \in \N$ such that $a_n \le C_p n^{-p}$ for all $n \ge N_p$, so we get $a_n^t \le C_p^t n^{-pt}$, for all $n \ge N_p$. Since $pt > 1$, the series $\sum_{n\ge1} n^{-pt}$ converges and, by comparison, $P_1(t)$ also converges.
    
    \medskip\noindent
    \emph{Divergence for $0 \le t \le 1/q$:}
    
    We analyze based on the finiteness of $q$:
    \begin{itemize}
        \item If $q < +\infty$, \ref{hypothesis:h3} provides a global constant $C \ge 1$ such that $a_n \ge (C n^q)^{-1}$ for all sufficiently large $n$. At $t = 1/q$, we have $a_n^{1/q} \ge C^{-1/q} n^{-1}$. Since the harmonic series diverges, $P_1(1/q) = +\infty$. For any $0 \le t < 1/q$ we have $a_n^t \ge a_n^{1/q}$ for large $n$, so it also diverges.
        \item If $q = +\infty$, our convention $1/q = 0$ means the interval $[0, 1/q] = \set{0}$, and then the calculation is just $P_1(0) = \sum_{n=1}^\infty a_n^0 = \sum_{n=1}^\infty 1 = +\infty$.
    \end{itemize}
    
    \medskip\noindent
    \emph{Limit for $t\to1/q^+$:}

If $q < +\infty$, let $\epsilon>0$ be a sufficiently small number such that $t = 1/q+\epsilon < 1$; since $C\ge 1, q>1$, this implies $\dfrac{1}{C^{1/q+\epsilon}} \ge \dfrac{1}{C}$.

Then we proceed with the following calculation
	\[
    \sum_{n\ge1}a_n^t \ge \sum_{n\ge1}\left(\dfrac{1}{Cn^q}\right)^{1/q+\epsilon} \ge \sum_{n\ge1}\dfrac{1}{Cn^{1+q\epsilon}} \to +\infty,
    \]
as $\epsilon\to0$, implying the desired limit.

If $q = +\infty$, then $1/q = 0$, and since the terms $a_n$ are strictly positive, we have $a_n^t \to 1$ as $t \to 0$, so the series diverges as $+\infty$ as $t\to0$.
\end{proof}

Now, as mentioned before, we are going to prove that

\begin{lemma}\label{lemma:exponential_decay_measure}
    Suppose $\Psi$ satisfies the Bounded Distortion Property \ref{condition:bdp} with constant $K \ge 1$ and the open gap set $G \subset [-1,1]$ (as in \eqref{eq:G}) has positive Lebesgue measure $\mu(G) > 0$. Then, there exists a constant $0 < \rho < 1$ such that
    \[
        \sum_{\omega \in \cal{J}^n} \mu(\psi_\omega[-1,1]) \le 2(1-\rho)^{n-1},
    \]
    for all $n\ge1$.
\end{lemma}

\begin{proof}
    Let us consider the ambient interval $X = [-1,1]$ and fix some $\omega \in \cal{J}^{n}$.
    
    By the Mean Value Theorem, we have the lower bound
    \[
    \mu(\psi_\omega(G)) \ge \mu(G) \inf_{z\in X} \abs{\psi_\omega'(z)}.
    \]
    
    Similarly, we have an upper bound
    \[
    \mu(\psi_\omega(X)) \le \mu(X) \sup_{z\in X} \abs{\psi_\omega'(z)} = 2 \sup_{z\in X} \abs{\psi_\omega'(z)}.
    \]
    Applying \ref{condition:bdp}, we have $\sup \abs{\psi_\omega'} \le M \inf \abs{\psi_\omega'}$, so we can deduce
    \begin{equation}\label{eq:bound_psi_G}
    \mu(\psi_\omega(G)) \ge \mu(G)\cdot\inf_{z\in X}\abs{\psi_\omega'(z)} \ge \mu(G)\dfrac{\sup_{z\in X}\abs{\psi_\omega'(z)}}{M} \ge \frac{\mu(G)}{M} \dfrac{\mu(\psi_\omega(X))}{2}.	
    \end{equation}

    Setting $\rho = \frac{\mu(G)}{2M}$, we have $0 < \rho < 1$.
    
    Now, we have 
    \[
        \sum_{j \in \cal{J}} \mu(\psi_{\omega j}(X)) \le \mu(\psi_\omega(X)) - \mu(\psi_\omega(G)) \le (1 - \rho)\mu(\psi_\omega(X)).
    \]
    The first inequality occurs because $\set{\psi_{\omega j}(X)}_{j\in\cal{J}}$ is a collection of disjoint sets contained in $\psi_\omega(X)\setminus\psi_\omega(G)$, and the second inequality is an application of \eqref{eq:bound_psi_G}.
    
    Summing over all words $\omega \in \cal{J}^{n}$ and applying this relation inductively starting from $n=1$, where $\sum_{J\in\cal{J}}\mu(\psi_{J}(X)) \le 2$, we get
    \[
    \sum_{\omega \in \cal{J}^n} \mu(\psi_\omega(X)) \le 2(1-\rho)^{n-1}.
    \]
\end{proof}

\begin{corollary}\label{corollary:exponential_bounds}
For all $n\ge1$ we have
\[
P_n(1) := \sum_{\omega \in \cal{J}^n} \sup_{z\in[-1,1]}\abs{\psi_{\omega}'(z)} \le M(1-\rho)^{n-1}.
\]
In particular, there exists $N\ge1$ such that
\[
P_N(1) < 1.
\]
\end{corollary}

\begin{proof}
We simply note that, using \ref{condition:bdp} and Lemma~\ref{lemma:exponential_decay_measure}, we have
\begin{align*}
P_n(1)
&=\sum_{\omega \in \cal{J}^n} \sup_{z\in[-1,1]}\abs{\psi_{\omega}'(z)} \\
&\le \sum_{\omega \in \cal{J}^n} M\inf_{z\in[-1,1]}\abs{\psi_{\omega}'(z)} \\
&\le M \sum_{\omega \in \cal{J}^n} \dfrac{\mu(\psi_\omega[-1,1])}{\mu[-1,1]} \\
&\le \dfrac{M}{2}2(1-\rho)^{n-1} \\
&= M(1-\rho)^{n-1}.
\end{align*}

This expression converges to $0$ as $n \to+\infty$, then there must exist an integer $N \ge 1$ such that $P_N(1) < 1$. 
\end{proof}

\section{Complexity of the singular set $Z_1$.} \label{section:appendixB}

The purpose of this appendix is to illustrate that the singular set $Z_1$ can exhibit a wide variety of topological structures.

\subsection{Example where $Z_1$ is a triadic Cantor set}\label{example:cantor_gaps}
    Consider an infinite-piecewise expanding map defined on $W \subset [-1,1]$ whose inverse branches $\Psi = \set{\psi_n}_{n\ge1}$ are affine contractions. We group these branches into blocks indexed by $k \in \N$, such that the $k$-th block contains $2^{k-1}$ branches, and each branch in this block has a constant derivative $1/r^k$ for some $r > 3$. In this specific example, each block is located in the gaps of the $k$-th step in the construction of the Cantor set (an illustration can be found in Figure~\ref{fig:Z_1_cantor_set}).
    
    We define $\Lambda(r)$ as the invariant set resulting of this construction with $r>3$.
    
    To verify that this system satisfies the polynomial decay condition in \ref{hypothesis:h3}, consider $F_n$ in the $k$-th block, so $2^{k-1} \le n < 2^k$, which implies
    \begin{equation*}
        \dfrac{\log n}{\log 2} < k \le \dfrac{\log n}{\log 2} + 1.
    \end{equation*}
    The derivative of the functions in this block is $\sup \abs{\psi_n'} = r^{-k}$. Using our bounds for $k$, we can bound the derivative by
    \begin{equation*}
        r^{-\left(\frac{\log n}{\log 2} + 1\right)} \le r^{-k} < r^{-\frac{\log n}{\log 2}}.
    \end{equation*}
    By observing that $r^{-\frac{\log n}{\log 2}} = n^{-\frac{\log r}{\log 2}}$, this simplifies to
    \begin{equation*}
        \dfrac{1}{r n^{\frac{\log r}{\log 2}}} \le \sup \abs{\psi_n'} < \dfrac{1}{n^{\frac{\log r}{\log 2}}}.
    \end{equation*}
    
    Thus, the system satisfies the condition of \ref{hypothesis:h3} with parameter $q = \dfrac{\log r}{\log 2}$, lower-bound constant $C = r$, and upper-bound constant $C_q = 1$. The convergence condition in \eqref{equation:convergence_of_series} is therefore satisfied for all $t > 1/q = \dfrac{\log 2}{\log r}$. Also, Theorem \ref{theorem:geometry_of_invariant_set} guarantees that $\hausdorff{\Lambda(r)} > \dfrac{\log 2}{\log r}$.

    By directly applying \cite[Theorem 9.3]{Falconer2006FractalGeometryMathematical} and \cite[Theorem 3.15]{MauldinUrbanski1996DimensionsMeasuresInfinite}, since the functions are affine (and therefore similarities), the established theory of CIFS guarantees that the Hausdorff dimension is given by the unique solution (in $t$) to
    \begin{equation*}
        \sum_{n=1}^\infty \abs{\psi_n'}^t = 1.
    \end{equation*}
    
    Since the $k$-th block has $2^{k-1}$ branches, each contributing $(r^{-k})^t$, the equation becomes:
    \begin{equation*}
        \sum_{k=1}^\infty 2^{k-1} (r^{-kt}) = \dfrac{1}{2} \sum_{k=1}^\infty \left(2r^{-t} \right)^k = 1.
    \end{equation*}
    
    This series converges for $2r^{-t} < 1 \iff t > \dfrac{\log 2}{\log r}$. By applying the infinite geometric series formula, we obtain
    \begin{equation*}
        \dfrac{1}{2} \left(\dfrac{2r^{-t}}{1 - 2r^{-t}} \right) = 1,
    \end{equation*}
    which implies $t = \dfrac{\log 3}{\log r}$.

    So for this case, we have exactly $\hausdorff{\Lambda(r)} = \dfrac{\log 3}{\log r}$. Since $r > 3$, we easily verify that $\dfrac{\log 2}{\log r} < \hausdorff{\Lambda(r)} < 1$, exactly as Theorem \ref{theorem:geometry_of_invariant_set} assures.
    
    Note that we have $\hausdorff{\Lambda(r)}\to 0$ as $r\to+\infty$. In this case, we can have $\hausdorff{\Lambda(r)} < \hausdorff{Z_1} = \log2/\log3$ (this last value being the Hausdorff dimension of the standard triadic Cantor set).

\begin{figure}[h!]
\centering 
\begin{overpic}[width=0.7\linewidth]{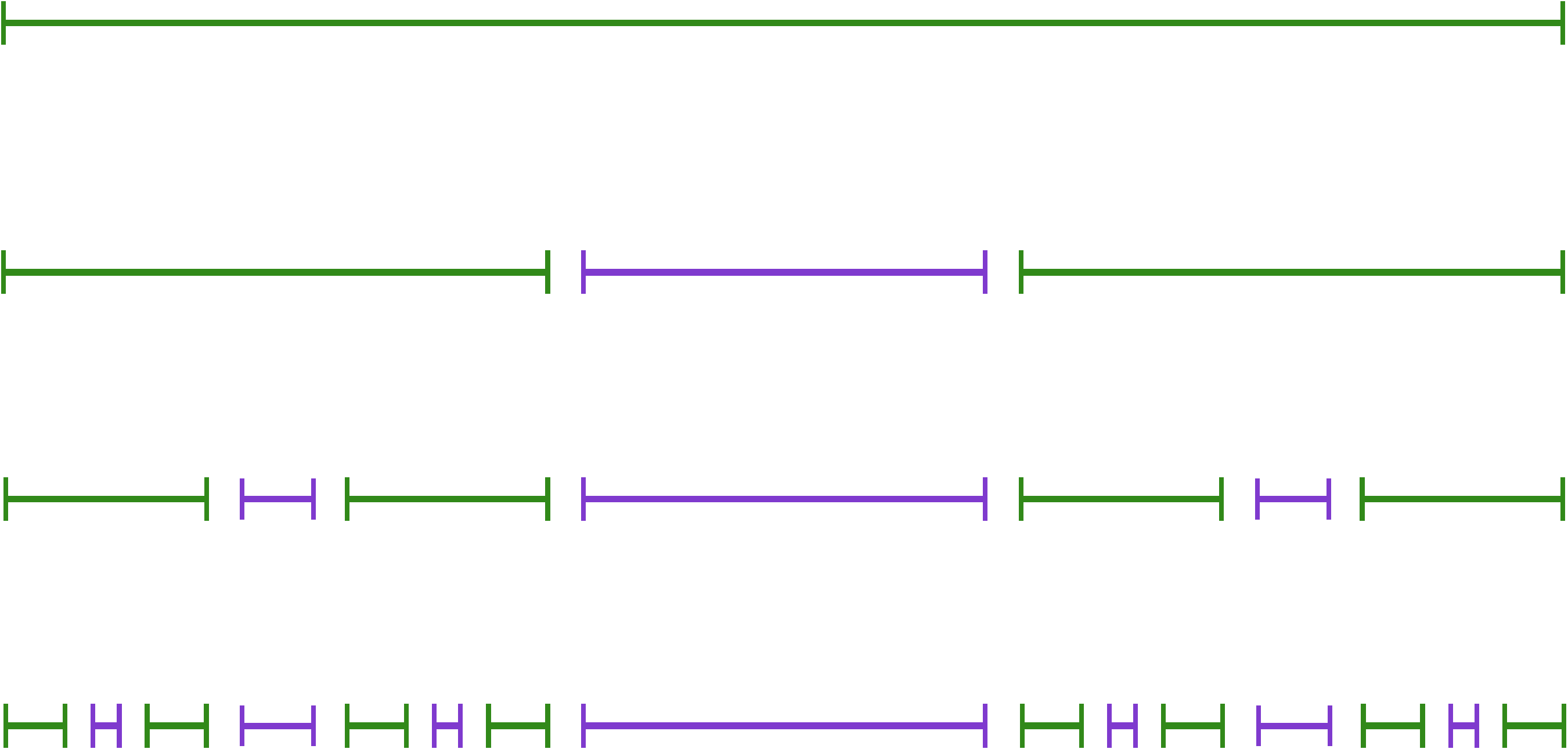}
\put(47.5,32){$F_1$}
\put(48,27){$r$}
\put(47.5,18){$F_1$}
\put(16,18){$F_2$}
\put(16.5,12){$r^2$}
\put(81,18){$F_3$}
\put(81.5,12){$r^2$}
\put(47.5,4){$F_1$}
\put(16,4){$F_2$}
\put(81,4){$F_3$}
\put(5,4){$F_4$}
\put(5.5,-3){$r^3$}
\put(27,4){$F_5$}
\put(27.5,-3){$r^3$}
\put(70,4){$F_6$}
\put(70.5,-3){$r^3$}
\put(92,4){$F_7$}
\put(92.5,-3){$r^3$}
\put(47.5,-7){$\ldots$}
\end{overpic}
\vspace{0.8cm}
\caption{Example with $Z_1$ being the Cantor set, detailed in Example \ref{example:cantor_gaps}}
\label{fig:Z_1_cantor_set}
\end{figure}
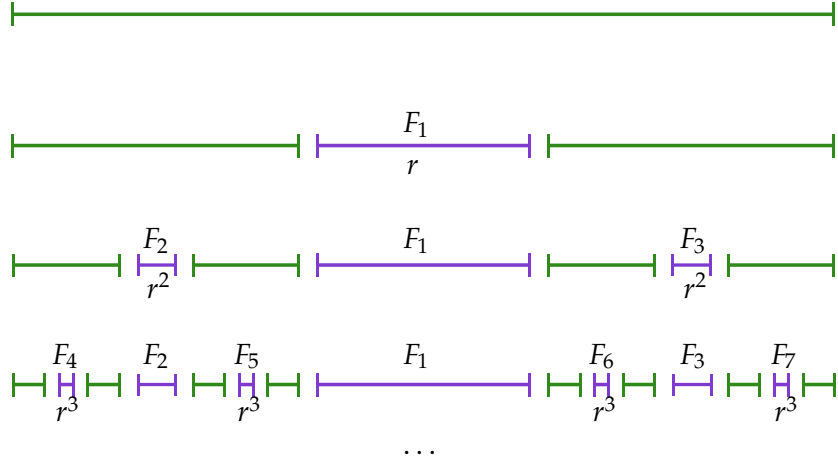

\subsection{Example where $Z_1$ is a fat triadic Cantor set}\label{example:cantor_fat}
We can generalize further this construction to obtain examples of $Z_1$ with non-zero Hausdorff dimension and non-zero Lebesgue measure. To see that, consider the process just made, but with fat Cantor sets, in order to get a $Z_1$ set with positive Lebesgue measure and $\hausdorff{Z_1}=1$. Besides, by making in this construction the intervals of $W$ as small as desired and taking the functions on $\Psi$ with rapidly enough growth slopes, we can have these kind of occurrences while also getting $\hausdorff{\Lambda}$ as small as desired and $\mu(\Lambda)=0$, showing that the maximum in Theorem \ref{theorem:geometry_of_closure_of_lambda} is really needed and that $\mu(\cl{\Lambda}) > \mu(\Lambda)$ are in fact realizabe in Theorems \ref{theorem:geometry_of_invariant_set} and \ref{theorem:geometry_of_closure_of_lambda}.

\subsection{Countably stable properties of $Z_1$}\label{example:structure_of_Z1}
At last, if $Z_1$ possesses a property that is invariant under Lipschitz functions and stable under countable unions, then $Z$ inherits this property. Notable examples include countability, Hausdorff dimension and null Lebesgue measure. In particular, this inheritance mechanism allowed for the more precise values of the cardinality and Lebesgue measure obtained in \cite[Theorem A]{CunhaNovaesPonce2024HausdorffDimensionCantor}, compared to the bounds established in Theorem \ref{theoremA}.

\section*{Acknowledgments}

MGCC was partially supported by S\~{a}o Paulo Research Foundation (FAPESP) grant 2025/01340-7. DDN was partially supported by the São Paulo Research Foundation (FAPESP), grants 2024/15612-6 and 2026/03312-3; by the Conselho Nacional de Desenvolvimento Científico e Tecnológico (CNPq), grant 301878/2025-0; and by the Coordenação de Aperfeiçoamento de Pessoal de Nível Superior - Brasil (CAPES), through the MATH-AmSud program, grant 88881.179491/2025-01. GP was partially supported by S\~{a}o Paulo Research Foundation (FAPESP) grants 2022/07762-2 and 2018/13481-0.

\bibliography{../references.bib}

\begin{thebibliography}{1}

\bibitem{CunhaNovaesPonce2024HausdorffDimensionCantor}
M.~G.~C. Cunha, D.~D. Novaes, and G.~Ponce.
\newblock On the {{Hausdorff}} dimension and {{Cantor}} set structure of sliding {{Shilnikov}} invariant sets.
\newblock {\em Nonlinearity}, 37(12):125023, Dec. 2024.

\bibitem{Falconer2006FractalGeometryMathematical}
K.~J. Falconer.
\newblock {\em Fractal Geometry: Mathematical Foundations and Applications}.
\newblock Wiley, Chichester, 2. ed., repr edition, 2006.

\bibitem{Filippov1988DifferentialEquationsDiscontinuous}
A.~F. Filippov.
\newblock {\em Differential {{Equations}} with {{Discontinuous Righthand Sides}}}, volume~18 of {\em Mathematics and {{Its Applications}}}.
\newblock Springer Netherlands, Dordrecht, 1988.

\bibitem{Mauldin1995InfiniteIteratedFunction}
R.~D. Mauldin.
\newblock Infinite {{Iterated Function Systems}}: {{Theory}} and {{Applications}}.
\newblock In C.~Bandt, S.~Graf, and M.~Z{\"a}hle, editors, {\em Fractal {{Geometry}} and {{Stochastics}}}, pages 91--110. Birkh\"auser Basel, Basel, 1995.

\bibitem{MauldinUrbanski1996DimensionsMeasuresInfinite}
R.~D. Mauldin and M.~Urba{\'n}ski.
\newblock Dimensions and {{Measures}} in {{Infinite Iterated Function Systems}}.
\newblock {\em Proceedings of the London Mathematical Society}, s3-73(1):105--154, July 1996.

\bibitem{NovaesTeixeira2019ShilnikovProblemFilippov}
D.~D. Novaes and M.~A. Teixeira.
\newblock Shilnikov problem in {{Filippov}} dynamical systems.
\newblock {\em Chaos: An Interdisciplinary Journal of Nonlinear Science}, 29(6):063110, June 2019.

\bibitem{PacificoRovellaViana1998InfinitemodalMapsGlobal}
M.~J. Pacifico, A.~Rovella, and M.~Viana.
\newblock Infinite-modal maps with global chaotic behavior.
\newblock 1998.

\bibitem{Pugh2015RealMathematicalAnalysis}
C.~C. Pugh.
\newblock {\em Real {{Mathematical Analysis}}}.
\newblock Undergraduate {{Texts}} in {{Mathematics}}. Springer International Publishing, Cham, 2015.

\bibitem{Schleicher2007HausdorffDimensionIts}
D.~Schleicher.
\newblock Hausdorff {{Dimension}}, {{Its Properties}}, and {{Its Surprises}}.
\newblock {\em The American Mathematical Monthly}, 114(6):509--528, June 2007.

\end{thebibliography}
\bibliographystyle{abbrv}

\end{document}